\documentclass[12pt]{amsart}
\usepackage[all]{xy}
\usepackage{graphicx}
\usepackage{tikz}  

\usepackage{amsmath}
\usepackage{amssymb}
\usepackage{amsfonts}
\usepackage{mathrsfs}
\usepackage{mathtools}

\newcommand{\Cdb}{\ensuremath{\mathbb{C}}}

\newcommand{\Ndb}{\ensuremath{\mathbb{N}}}

\newcommand{\Rdb}{\ensuremath{\mathbb{R}}}

\newcommand{\Tdb}{\ensuremath{\mathbb{T}}}

\newcommand{\Zdb}{\ensuremath{\mathbb{Z}}}

\newcommand{\A}{\mbox{${\mathcal A}$}}

\newcommand{\E}{\mbox{${\mathcal E}$}}

\newcommand{\M}{\mbox{${\mathcal M}$}}

\renewcommand{\P}{\mbox{${\mathcal P}$}}

\newcommand{\T}{\mbox{${\mathcal T}$}}

\newcommand{\norm}[1]{\Vert#1\Vert}
\newcommand{\bignorm}[1]{\bigl\Vert#1\bigr\Vert}
\newcommand{\Bignorm}[1]{\Bigl\Vert#1\Bigr\Vert}

\newtheorem{theorem}{Theorem}[section]
\newtheorem{lemma}[theorem]{Lemma}
\newtheorem{corollary}[theorem]{Corollary}
\newtheorem{proposition}[theorem]{Proposition}
\newtheorem{definition}[theorem]{Definition}
\theoremstyle{remark}
\newtheorem{remark}[theorem]{\bf Remark}
\theoremstyle{definition}

\numberwithin{equation}{section}

\begin{document}

\title[$\gamma$-bounded representations]
{A new approach to $\gamma$-bounded representations}

\author[C. Le Merdy]{Christian Le Merdy}
\email{clemerdy@univ-fcomte.fr}
\address{Laboratoire de Math\'ematiques de Besan\c con, 
Universit\'e de Franche-Comt\'e, 16 route de Gray
25030 Besan\c{c}on Cedex, FRANCE}

\date{\today}

\maketitle

\begin{abstract}
Let $X$ be a Banach space, let $(\Omega,\mu)$ 
be a $\sigma$-finite measure space and let $A,B\colon\Omega\to B(X)$
be strongly measurable $\gamma$-bounded functions. We show that
for  all $x\in X$ and all $x^*\in X^*$, there exist a Hilbert space $K$ and two measurable
functions 
$a_1\in L^\infty(\Omega;K)$ and $a_2\in L^\infty(\Omega;K)$  such that 
$\langle B(t)A(s)x,x^*\rangle = (a_2(t)\,\vert\, a_1(s))_{K}$ for a.e.
$(s,t)$ in $\Omega^2$, with
$\norm{a_1}_\infty\norm{a_2}_\infty\leq \gamma(A)\gamma(B)
\norm{x}\norm{x^*}$. This factorization property allows us to improve or simplify some
results concerning $\gamma$-bounded representations of groups or semigroups.

\end{abstract}

\vskip 1cm
\noindent
{\it 2000 Mathematics Subject Classification :}  47A99, 22D12, 47D06, 47A60

\smallskip
\noindent
{\it Key words:} 
$\gamma$-boundedness, factorizations, group representations, semigroups, functional calculus.

\vskip 1cm

\section{Introduction}\label{Intro}
In the last 25 years, the concepts of $\gamma$-boundedness and 
$R$-boundedness played a prominent role in the development of vector valued harmonic analysis 
and operator theory on Banach spaces.  It started with Weis's characterization of $L^p$-maximal regularity
for sectorial operators acting on UMD Banach spaces \cite{Weis} and after that,  $\gamma$-boundedness and 
$R$-boundedness had a huge impact in the theory of $H^\infty$-functional calculus, see  
\cite{KW1, KW0, KW} as well as \cite{HVVW} and the references therein. Besides this,
$\gamma$-boundedness played
a key role in control theory (see e.g. \cite{HK}), vector valued stochastic integration
(see e.g. \cite{VVW}) and Banach spaces representations. The latter is the 
main theme of this note.

Section \ref{Gamma2} is devoted to an abstract factorization result (Theorem \ref{2Main})
which can be stated as follows. Let $X$ be a Banach space, let $(\Omega,\mu)$ 
be a $\sigma$-finite measure space and consider two strongly measurable functions
$A,B\colon\Omega\to B(X)$. Assume that $A$ and $B$ are $\gamma$-bounded, that is,  
$\{A(t)\, :\, t\in\Omega\}$ and $\{B(t)\, :\, t\in\Omega\}$ are $\gamma$-bounded
subsets of $B(X)$.
Then for all $x\in X$ and all $x^*\in X^*$, there exist a Hilbert space $K$, with inner product
$(\,\cdotp\,\vert\, \,\cdotp)_{K}$, and two measurable
functions 
$a_1\in L^\infty(\Omega;K)$ and $a_2\in L^\infty(\Omega;K)$  such that 
\begin{equation}\label{2Factor}
\bigl\langle B(t)A(s)x,x^*\bigr\rangle = (a_2(t)\,\vert\, a_1(s))_{K}\qquad\hbox{for\ 
a.e.-}(s,t)\in\Omega^2,
\end{equation}
and 
$\norm{a_1}_\infty\norm{a_2}_\infty\leq \gamma(A)\gamma(B)
\norm{x}\norm{x^*}.$
Here $\gamma(\,\cdotp)$ stands for the $\gamma$-boundedness constant, see Definition \ref{Gamma-app}.

This relationship between $\gamma$-boundedness and Hilbert space factorizations is not so surprising.
Indeed when $X$ is a Hilbert space, bounded subsets of $B(X)$ are automatically $\gamma$-bounded,
see Remark \ref{2Hilb} below, and quite often, results involving a $\gamma$-boundedness assumption
are Banach space generalizations of Hilbertian results requiring the boundedness of a certain 
set of operator. See also \cite[Section 2.1]{KW1} for more on this phenomenon.

In Section 3, we apply Theorem \ref{2Main} to cases when $\Omega$ is either
a group or a semigroup, and  $A=B\colon \Omega\to B(X)$
is a multiplicative map (= a representation) which is assumed
$\gamma$-bounded. Indeed, Theorem \ref{2Main} seems to be particularly
suitable for this framework. We either improve or simplify 
results from \cite{A, AL, ALZ, LM1}.
Here is a more precise description.

Firstly, let $G$ be a locally compact group. Let $\pi\colon G\to B(X)$ be a strongly continuous
bounded representation and let $\sigma_\pi\colon L^1(G)\to B(X)$ be the associated 
bounded homomorphism defined by $\sigma_\pi(f)=\int_Gf(t)\pi(t)\,dt$.
The main result of \cite{LM1} asserts that if $\pi$ is $\gamma$-bounded and $G$ is amenable, then
$\sigma_\pi$ extends to a bounded homomorphism $w_\pi\colon C_\lambda(G)\to B(X)$, with
$\norm{w_\pi}\leq\gamma(\pi)^2$. Here
$C_\lambda(G)$ denotes the reduced $C^*$-algebra of $G$. In Subsection \ref{Group},
we give a new, simple proof of this result and extend it to non-amenable groups.

Secondly, let $\A_{0,S^1}(\Rdb)\subset C_0(\Rdb)$ be the function algebra introduced in \cite{ALZ}.
In the latter paper, we proved that all bounded $C_0$-semigroups on Hilbert space admit a bounded functional
calculus $\A_{0,S^1}(\Rdb)\to B(H)$ which extends the Hille-Phillips functional
calculus. In Subsection \ref{Semigroup1}, we generalize this result to Banach spaces by showing that 
if $\T=(T_t)_{t\geq 0}$ is a $\gamma$-bounded $C_0$-semigroup on $X$, then there
exists a bounded homomorphism 
$\rho_{\mathcal T} \colon \A_{0,S^1}(\Rdb)\longrightarrow B(X)$
such that $\rho_{\mathcal T}(\widehat{b}(-\,\cdotp)) =
\int_{0}^\infty b(t)T_t\,dt\,$
for all $b\in L^1(\Rdb_+)$. 

Thirdly, in Subsection \ref{PB} we give a 
discrete analogue of the above result,
which slightly improves \cite[Section 4.5]{A}. More precisely we introduce an analogue
$\A_{0,S^1}(\Tdb)\subset C(\Tdb)$ of $\A_{0,S^1}(\Rdb)$ and we show that
if $T\in B(X)$ is a power $\gamma$-bounded operator,  then there
exists a bounded homomorphism 
$\rho_{T} \colon \A_{0,S^1}(\Tdb)\longrightarrow B(X)$
such that $\rho_{T}(F) = \sum_{n=0}^\infty\widehat{F}(n)T^n$
for all polynomial $F(z)=\sum_{n=0}^\infty\widehat{F}(n)z^n.$

\section{$\gamma$-boundedness and factorizations}\label{Gamma2}

\subsection{Main statement}\label{MS} 
Let $X$ be a Banach space and let $X^*$ be its dual space. We let $\langle x,x^*\rangle$ denote the action
of any $x^*\in X^*$ on  $x\in X$. If $Y$ is another Banach space, we let $B(Y,X)$ denote 
the Banach space of all bounded operators 
from $Y$ into $X$. When $Y=X$, we simply write $B(X)$ instead of $B(X,X)$.

Let $(\Omega,\mu)$ be a $\sigma$-finite measure space and let $1\leq p\leq \infty$.
We let  $L^p(\Omega;X)$ denote the Bochner space 
of all measurable functions $F\colon \Omega\to X$
(defined up to almost everywhere zero functions)  such that the
norm function $t\mapsto \norm{F(t)}$ belongs to $L^p(\Omega)$.
This is a Banach space for the norm $\norm{\,\cdotp}_p$ defined by
$$
\norm{F}_p = \Bigl(\int_\Omega \norm{F(t)}^p\, d\mu(t)\Bigr)^\frac1p
$$
if $p$ is finite, and by $\norm{F}_\infty={\rm supess}_t\norm{F(t)}$ if $p=\infty$.
We will use the fact that $L^p(\Omega)\otimes X$
is dense in $L^p(\Omega;X)$ if $p$ is finite. We refer  to \cite[Chapters I and II]{DU}
for information on these spaces. In the sequel we will use them only
for $p\in\{1,2,\infty\}$.

Now consider two $\sigma$-finite measure spaces $(\Omega_1,\mu_1)$ 
and $(\Omega_2,\mu_2)$. We let
$\mathcal{V}_2(\Omega_1\times\Omega_2)$ be the set of 
all $\phi\in  L^\infty(\Omega_1\times\Omega_2)$ for which
there exist a Hilbert space $K$ and two functions
$a_1\in L^\infty(\Omega_1;K)$ and $a_2\in L^\infty(\Omega_2;K)$  such that 
\begin{equation}\label{2Factor}
\phi(s,t) = (a_2(t)\,\vert\, a_1(s))_{K}\qquad\hbox{for\ 
a.e.-}(s,t)\in\Omega_1\times\Omega_2.
\end{equation}
Here $(\,\cdotp\,\vert\, \,\cdotp)_{K}$ stands for the inner product on $K$.
For any $\phi$ as above, we set
$$
\nu_2(\phi)\,=\,\inf\bigl\{\norm{a_1}_{\infty}
\norm{a_2}_{\infty}\bigr\},
$$
where the infimum runs over all factorizations
of $\phi$ of the form (\ref{2Factor}).
It turns out that $\mathcal{V}_2(\Omega_1\times\Omega_2)$ is a vector space
and that $\nu_2(\,\cdotp)$ is a norm. We refer the reader to 
\cite[Subsection 2.1]{ALZ}
for information on this space. (See also Subsection \ref{Schur} below.)

We recall the definition of $\gamma$-boundedness.  Let
$(\gamma_j)_{j\geq 1}$ be a sequence of complex valued, independent
standard Gaussian variables on some probability space $(\Sigma,\nu)$,
and consider $G_0={\rm Span}\{\gamma_j\, :\, j\geq 1\}\subset L^2(\Sigma)$.
We let 
$$
G(X)\subset L^2(\Sigma;X)
$$ 
be the closure of $G_0\otimes X$,
equipped with the induced norm. We say that a set $\E\subset B(X)$ is $\gamma$-bounded if
there is a constant $C\geq 0$ such that for any finite families
$(V_1,\ldots, V_n)$ in $\E$, and $(x_1,\ldots,x_n)$ in $X$, we have
$$
\Bignorm{\sum_{j=1}^n \gamma_j\otimes V_j(x_j)}_{G(X)}\,\leq\, C\,
\Bignorm{\sum_{j=1}^n \gamma_j\otimes x_j}_{G(X)}.
$$
In this case, we let $\gamma(\E)$ denote the smallest possible $C$. It is usually
called the $\gamma$-boundedness constant of $\E$.
Obviously any  $\gamma$-bounded set is bounded.

\begin{remark}\label{2Hilb} 
If $X=H$ is a Hilbert space, then
$\bignorm{\sum_{j=1}^n \gamma_j\otimes x_j}_{G(X)}^2=\sum_j\norm{x_j}^2$
for all finite families $(x_j)_j$ in $H$. This implies that any bounded set 
$\E\subset B(H)$ is automatically $\gamma$-bounded, with 
$$
\gamma(\E)=\sup_{V\in \,{\mathcal E}}\norm{V}_{B(H)}.
$$
\end{remark}

\begin{definition}\label{Gamma-app} Let $I$ be a set and let $A\colon I\to B(X)$ be a 
function. We say that $A$ is $\gamma$-bounded if the range set 
$$
\E(A): =\{A(t)\, :\, t\in I\}\subset B(X)
$$
is $\gamma$-bounded. In this case, its $\gamma$-boundedness constant is defined by
$$
\gamma(A):=\gamma(\E(A)).
$$
\end{definition}

A function $A\colon\Omega
\to B(X)$ defined on a $\sigma$-finite measure space
is called  strongly measurable  if $t\mapsto A(t)x$ is measurable 
(from $\Omega$ into $X$) for all $x\in X$.

Our main result is the following. Its proof will be given in Subsection \ref{Pf}.

\begin{theorem}\label{2Main}
Let $(\Omega,\mu)$ be a $\sigma$-finite measure space and let $A,B\colon\Omega\to B(X)$
be two $\gamma$-bounded strongly measurable functions. Then for all
$x\in X$ and all $x^*\in X^*$, the function $\Theta_{x,x^*}\colon\Omega^2\to\Cdb$ defined by
\begin{equation}\label{2Theta}
\Theta_{x,x^*}(s,t) =\, \bigl\langle B(t)A(s)x,x^*\bigr\rangle
\end{equation}
belongs to $\mathcal{V}_2(\Omega^2)$, and we have
$$
\nu_2\bigl(\Theta_{x,x^*}\bigr)\leq 
\gamma(A)\gamma(B)\norm{x}\norm{x^*}.
$$
\end{theorem}

\medskip
\subsection{The Kalton-Weis $\gamma$-spaces}\label{G}
In order to study $\gamma$-boundedness, 
 Kalton and Weis introduced the so-called 
$\gamma$-spaces in the paper
\cite{KW}, which began to circulate about 20 years ago. 
We refer either to this reference
or to  \cite[Chapter 9]{HVVW} for a comprehensive exposition.
Here we  give a brief presentation and state 
two important results which will be used in the proof of Theorem \ref{2Main}.

Let $X$ be a Banach space and let $H$ be a Hilbert space. An operator $u\in B(H,X)$ is called $\gamma$-summing if 
$$
\norm{u}_\gamma : = 
\sup\Bigl\{\Bignorm{\sum_j \gamma_j\otimes u(e_j)}_{G(X)}\Bigr\}\,<\infty,
$$
where the supremum runs over all finite orthonormal systems $(e_j)_j$ in $H$.
We let $\gamma_\infty(H,X)$ denote the 
space of all $\gamma$-summing $u\in B(H,X)$.
Equipped with the norm $\norm{\, \cdotp}_\gamma$, 
$\gamma_\infty(H,X)$ is a Banach space.
Note that $\norm{u}_\gamma\geq \norm{u}_{B(H,X)}$.
Hence we have a contractive embedding
\begin{equation}\label{2Embed1}
\gamma_\infty(H,X)\subset B(H,X).
\end{equation}

We identify the algebraic tensor product $H^*\otimes X$ with the subspace of
$B(H,X)$ of all bounded finite rank operators in the usual way.
Namely for any finite families $(\eta_j)_j$ in $H^*$ and $(x_j)_j$ in
$X$, we identify the element $\sum_j\eta_j\otimes x_j$
with the operator $u\colon H\to X$ defined by letting
$u(h)=\sum_j \langle h,\eta_j\rangle x_j\,$, for all $h\in
H$. It turns out that any finite rank bounded operator $u\colon H\to X$ is
$\gamma$-summing. More precisely, we have the following easy
estimate whose proof is left to the reader.

\begin{lemma}\label{2Rk1}
For any $\eta\in H^*$ and $x\in X$, we have $\norm{\eta\otimes x}_\gamma=\norm{\eta}\norm{x}.$
\end{lemma}

We let 
$$
\gamma(H,X)\subset \gamma_\infty(H,X)
$$
denote the closure of $H^*\otimes X$ in $\gamma_\infty(H,X)$.

The following is a special case of \cite[Theorem 9.1.10]{HVVW}.

\begin{proposition}\label{2Tensor}
Let $S\in B(H)$ and let $u\in \gamma_\infty(H,X)$. Then the composition 
operator $u\circ S\colon H\to X$ belongs to $\gamma_\infty(H,X)$ and we have 
$\norm{u\circ S}_\gamma\leq  \norm{S}_{B(H)}\norm{u}_\gamma$.
\end{proposition}

We now focus on the case when $H=L^{2}(\Omega)$, for some
$\sigma$-finite measure space $(\Omega,\mu)$. We identify $L^{2}(\Omega)^*$ and
$L^{2}(\Omega)$ in the usual way, using integration. Thus we both have
\begin{equation}\label{2Embed2}
L^2(\Omega)\otimes X\subset L^2(\Omega;X)
\qquad\hbox{and}\qquad L^2(\Omega)\otimes X\subset \gamma(L^2(\Omega),X).
\end{equation}
Applying (\ref{2Embed1}), we further obtain a contractive
embedding 
\begin{equation}\label{2Embed3}
\gamma_\infty(L^2(\Omega),X)\subset B(L^2(\Omega),X).
\end{equation}
Also, there is a natural contractive
embedding 
\begin{equation}\label{2Embed4}
L^{2}(\Omega; X)\subset B(L^2(\Omega),X)
\end{equation}
obtained by
identifying any $F\in L^{2}(\Omega; X)$ with the operator 
$u_F\colon L^2(\Omega)\to X$ defined by
$$
u_F(h) = \int_{\Omega} F(t)h(t)\,d\mu(t)\,,\qquad
h\in L^{2}(\Omega).
$$
This integral makes sense since $Fh\in L^1(\Omega;X)$ if $h\in L^2(\Omega)$.

It is easy to check that the embeddings provided
by (\ref{2Embed2}), (\ref{2Embed3}) and (\ref{2Embed4}) are all compatible 
with each other. Since $\gamma_\infty(L^2(\Omega),X)$ and
$L^{2}(\Omega; X)$ are both regarded as subspaces of $B(L^2(\Omega),X)$,
it makes sense to consider the intersections  $\gamma_\infty(L^2(\Omega),X)\cap L^2(\Omega;X)$
and $\gamma(L^2(\Omega),X)\cap L^2(\Omega;X)$. In this respect,
an element $F$ of $\gamma(L^2(\Omega),X)\cap L^2(\Omega;X)$ (resp. 
$\gamma_\infty(L^2(\Omega),X)\cap L^2(\Omega;X)$) is simply a function
$F\in  L^2(\Omega;X)$ such that $u_F$ belongs to 
$\gamma(L^2(\Omega),X)$ (resp. 
$\gamma_\infty(L^2(\Omega),X)$).

Consider a bounded, strongly measurable function
$A\colon\Omega\to B(X)$.
For any $F\in L^2(\Omega;X)$, let 
$AF\colon\Omega\to X$ be defined by 
$$
(AF)(t)=[A(t)](F(t)), \qquad t\in\Omega.
$$
It is clear that $AF$ is measurable if $F$ belongs to
$L^2(\Omega)\otimes X$. Since the latter is dense in 
$L^2(\Omega;X)$, we  obtain that $AF$ is measurable
for all $F$ in 
$L^2(\Omega;X)$. Finally, the boundedness of $A$ ensures that 
$AF\in L^2(\Omega;X)$ for all $F$ in 
$L^2(\Omega;X)$. This allows to define
\begin{equation}\label{2MA}
M_A\colon L^2(\Omega;X)\longrightarrow L^2(\Omega;X),\qquad M_A(F)=AF.
\end{equation}

The following $\gamma$-multiplier theorem is a slight
variant of \cite[Theorem 9.5.1]{HVVW}), see also \cite[Proposition 3.5]{LM1}.

\begin{proposition}\label{2Mult}
Let $A\colon\Omega\to B(X)$ be a strongly
measurable function and assume that $A$ is $\gamma$-bounded.
For any $F\in \gamma(L^2(\Omega),X)\cap L^2(\Omega;X)$, 
$M_A(F)$ belongs to $\gamma_\infty(L^2(\Omega),X)$ and we have
$$
\norm{M_A(F)}_\gamma\leq\gamma(A)\norm{F}_\gamma.
$$
\end{proposition}

\medskip
\subsection{Measurable Schur multipliers}\label{Schur} 
We present a key relationship between the 
${\mathcal V}_2$-spaces introduced in Subsection \ref{MS} and   Schur multipliers.

For any integer $n\geq 1$, we let $M_n$
denote the algebra of all $n\times n$ matrices and we identify $M_n$ with 
$B(\ell^2_n)$ in the usual way. For any family $\phi=\{\phi(k,l)\}_{1\leq k,l\leq n}$
of complex numbers, the Schur multiplication operator 
$T_\phi\colon B(\ell^2_n)\to B(\ell^2_n)$ is defined by
$$
T_\phi\bigl([c_{kl}]\bigr) = [\phi(k,l)c_{kl}],\qquad 
[c_{kl}]\in M_n.
$$
Next, let $I$ be any set and
let $\{e_{kl}\, :\, (k,l)\in I^2\}$ denote
the standard  matrix units on $B(\ell^2_I)$. Consider
$\phi=\{\phi(k,l)\}_{(k,l)\in I^2}\in \ell^\infty_{I^2}$. 
For any finite subset $J\subset I$, let $\phi_J=\{\phi(k,l)\}_{(k,l)\in J^2}$ be the 
restriction of this family to $J^2$. We say that  $\phi$ is a bounded Schur multiplier if 
there exists a constant $C\geq 0$ such that 
\begin{equation}\label{2C}
\norm{T_{\phi_J}\colon B(\ell^2_J)\to B(\ell^2_J)}
\leq C
\end{equation}
for all finite subsets $J\subset I$. It is plain that in this case, there exists a necessarily unique
$w^*$-continuous map $T_\phi\colon B(\ell^2_I)\to B(\ell^2_I)$ such that 
$T_\phi(e_{kl})=\phi(k,l)e_{kl}$ for all $(j,k)\in I^2$. Moreover the norm
$\norm{T_\phi\colon B(\ell^2_I)\to B(\ell^2_I)}$ is equal to the smallest possible $C\geq 0$ 
such that (\ref{2C}) holds true for all finite subsets $J\subset I$.

We now describe measurable 
Schur multipliers, following Haagerup \cite{Haa} and Spronk \cite{Spronk}.
Let $(\Omega,\mu)$ be a $\sigma$-finite measure space and let $S^2(L^2(\Omega))$ denote the
Hilbert space of all Hilbert-Schmidt operators on $L^2(\Omega)$.
For any  $\varphi\in L^2(\Omega^2)$, we let $S_\varphi\colon L^2(\Omega)\to L^2(\Omega)$
be the bounded operator defined by 
\begin{equation}\label{2S}
[S_\varphi(h)](s) = \int_\Omega \varphi(s,t)h(t)\, dt,
\qquad h\in L^2(\Omega).
\end{equation}
It is well-known that $S_\varphi\in S^2(L^2(\Omega))$ for all $\varphi\in
L^2(\Omega^2)$ and that the mapping 
$\varphi\mapsto S_\varphi$ yields a unitary identification
\begin{equation}\label{2HS}
L^2(\Omega^2)\simeq S^2(L^2(\Omega)),
\end{equation}
see e.g. \cite[Theorem VI. 23]{RS}.

Let $\phi\in L^\infty(\Omega^2)$. Using (\ref{2HS}), we define
$T_\phi\colon S^2(L^2(\Omega))\to S^2(L^2(\Omega))$
by setting 
$$
T_\phi(S_\varphi)=S_{\phi\varphi},\qquad
\varphi\in L^2(\Omega^2).
$$
Let $S^\infty(L^2(\Omega))$ denote the Banach space of all compact operators
on $L^2(\Omega)$ and recall that $S^2(L^2(\Omega))$ is a dense
subspace of $S^\infty(L^2(\Omega))$.
We say that $\phi$ is a bounded Schur multiplier if $T_\phi$ extends to a bounded
operator from $S^\infty(L^2(\Omega))$ into itself.  
In this case, 
$T_\phi$ admits a necessarily unique $w^*$-continuous extension from $B(L^2(\Omega))$ into
itself. In the sequel, 
we  keep the notation
$T_\phi\colon B(L^2(\Omega))\to B(L^2(\Omega))$
to denote this extension.

In the case when $\Omega=I$ is a discrete set equipped with the counting measure,
we have $L^2(\Omega)=\ell^2_I$ and $L^\infty(\Omega^2)=\ell^\infty_{I^2}$. Then 
the above defintions are equivalent to the ones given at the beginning of this subsection.

Bounded Schur multipliers can be described in terms of Hilbert space
characterizations.
In the discrete case, the following result  
was stated by Paulsen in \cite[Corolary 8.8]{Pau} and
by Pisier in \cite[Theorem 5.1]{Pis}.
For the general case, we refer to Haagerup
\cite{Haa} and Spronk \cite[Section 3.2]{Spronk}.

\begin{proposition}\label{2Description}
Consider a function $\phi\in L^\infty(\Omega^2)$. Then $\phi$  is a bounded Schur multiplier if and only if
$\phi\in{\mathcal V}_2(\Omega^2)$. In this case, we have
$$
\bignorm{T_\phi\colon B(L^2(\Omega))\longrightarrow B(L^2(\Omega))}
=\nu_2(\phi).
$$
\end{proposition}

\medskip
\subsection{Proof of Theorem \ref{2Main}}\label{Pf}
We consider two $\gamma$-bounded strongly measurable functions
$A,B\colon\Omega\to B(X)$.
Let $x\in X$ and $x^*\in X^*$ and let $ \Theta_{x,x^*}$ be defined by (\ref {2Theta}).
Let $\varphi\in L^2(\Omega)\otimes L^2(\Omega)$
and let $f,g\in L^2(\Omega)$. Our goal is to give an estimate of the integral
\begin{equation}\label{2J}
\Lambda:=\int_{\Omega^2} \Theta_{x,x^*}(s,t)\varphi(s,t) f(s)g(t)\, dsdt.
\end{equation}
Consider the operator $M_A$ 
induced by $A$, according to (\ref{2MA}). Using  (\ref{2Embed4}), 
we regard $M_A(f\otimes x)$ as an operator
$M_A(f\otimes x)\colon L^2(\Omega)\to X$. The latter is given by
$$
\bigl[M_A(f\otimes x)\bigr](k) =\int_\Omega k(s) f(s) A(s)x\, ds,\qquad k\in L^2(\Omega).
$$
Recall $S_\phi\colon L^2(\Omega)\to L^2(\Omega)$ given by (\ref{2S}).
Then by Fubini's theorem, we have
$$
\bigl[M_A(f\otimes x)\bigr]\bigl(S_\phi(h)\bigr)
=\int_{\Omega^2}\phi(s,t) h(t)f(s)A(s)x\, dtds,
\qquad h\in L^2(\Omega).
$$
We write $\phi=\sum_j\beta_j\otimes \alpha_j$, for some finite families
$(\alpha_j)_j$ and $(\beta_j)_j$ of $L^2(\Omega)$.  Then for any 
$h\in L^2(\Omega)$, we have
\begin{equation}\label{2Dev}
\bigl[M_A(f\otimes x)\bigr]\bigl(S_\phi(h)\bigr)
=\sum_j \Bigl(\int_\Omega \alpha_j(t)h(t)\,dt\Bigr)\Bigl(
\int_{\Omega}\beta_j(s)  f(s)A(s)x\, ds\Bigr).
\end{equation}
In particular, this shows that the operator
$\bigl[M_A(f\otimes x)\bigr]\circ S_\phi\colon
L^2(\Omega)\to X$ is finite rank, i.e. belongs 
to $L^2(\Omega)\otimes X$. We may therefore define
$$
\Gamma:= M_B\Bigl( \bigl[M_A(f\otimes x)\bigr]\circ S_\phi\Bigr)\colon L^2(\Omega) 
\longrightarrow
X,
$$
where again, we use the embedding (\ref{2Embed4}). Moreover by  Proposition \ref{2Mult}, 
$\Gamma$ belongs to the space
$\gamma_\infty(L^2(\Omega),X)$. 
Using  Proposition \ref{2Mult} and Proposition 
\ref{2Tensor}, we further  have
\begin{align*}
\norm{\Gamma}_\gamma 
& \leq \gamma(B)\bignorm{\bigl[M_A(f\otimes x)\bigr]\circ S_\phi}_\gamma\\
& \leq \gamma(B)\norm{S_\phi}_{B(L^2)}\norm{M_A(f\otimes x)}_\gamma\\
& \leq \gamma(A)\gamma(B)\norm{S_\phi}_{B(L^2)}\norm{f\otimes x}_\gamma.
\end{align*}
Finally using Lemma \ref{2Rk1}, we deduce that
$$
\norm{\Gamma}_\gamma\leq
\gamma(A)\gamma(B)\norm{S_\phi}_{B(L^2)}\norm{f}_2\norm{x}.
$$

It follows from (\ref{2Dev}) that as an element of $L^2(\Omega;X)$, 
$\Gamma$
is given by
$$
t\mapsto \sum_j \alpha_j(t) B(t)\Bigl(
\int_{\Omega}\beta_j(s)  f(s)A(s)x\, ds\Bigr).
$$
Consequently,
$$
\Gamma(g) =  \sum_j \int_\Omega \alpha_j(t) g(t) B(t)\Bigl(
\int_{\Omega}\beta_j(s)  f(s)A(s)x\, ds\Bigr)\, dt.
$$
Using again Fubini's theorem, we deduce that
\begin{align*}
\Gamma(g) 
& = \sum_j\int_{\Omega^2} 
\alpha_j(t) \beta_j(s) g(t) f(s) B(t) A(s)x\, ds dt\\
& = \int_{\Omega^2} 
\phi(s,t) g(t) f(s) B(t) A(s)x\, ds dt.
\end{align*}
Therefore,
$$
\langle \Gamma(g),x^*\rangle = \Lambda,
$$
where $\Lambda$ is given by (\ref{2J}). We deduce that
$\vert \Lambda\vert\leq \norm{g}_2\norm{x^*}\norm{\Gamma}_\gamma$, and hence
$$
\vert \Lambda\vert\leq \gamma(A)\gamma(B)\norm{x}\norm{x^*}
\norm{S_\phi}_{B(L^2)}\norm{f}_2\norm{g}_2.
$$
Since $L^2(\Omega)\otimes L^2(\Omega)$ is dense in
$L^2(\Omega^2)$, this implies that $\Theta_{x,x^*}$ is a bounded Schur multiplier
and that the norm of $T_{\Theta_{x,x^*}}\colon B(L^2(\Omega))\to B(L^2(\Omega))$
is less than or equal to $\gamma(A)\gamma(B)\norm{x}\norm{x^*}$.
Then the result follows from Proposition \ref{2Description}.

\section{Applications}\label{App}
Throughout this section, we let $X$ be an arbitrary Banach space.
We shall apply Theorem \ref{2Main} in the following cases:

\smallskip
- Either $\Omega$ is a locally compact group $G$ equiped with a left Haar measure, 
or $\Omega$ is the semigroup $\Rdb_+$ equiped with the Lebesgue measure, or
$\Omega$ is the semigroup $\Ndb=\{0,1,2,\ldots\}$ 
equiped with the counting measure. 

\smallskip
- The functions $A$ and $B$ are equal to a strongly continuous and $\gamma$-bounded 
multiplicative map $\Omega\to B(X)$.

\medskip
\subsection{Group representations}\label{Group}
Let $G$ be a locally compact group and let $dt$ denote a fixed left Haar measure on 
$G$. Let $\lambda\colon G\to B(L^2(G))$ denote the left regular representation of 
$G$, defined by 
$$
[\lambda(s)f](t)=f(s^{-1}t),\qquad t\in G,
$$
for all $s\in G$ and $f\in L^2(G)$. The group von Neumann algebra 
$VN(G)\subset B(L^2(G))$ is defined as the von Neumann algebra 
generated by $\{\lambda(s)\, :\, s\in G\}$.

A function $\psi\colon G\to \Cdb$ is called a bounded Fourier multiplier on 
$VN(G)$ if there exists a $w^*$-continuous operator
$M_\psi\colon VN(G)\to VN(G)$ such that $M_\psi(\lambda(s))=
\psi(s)\lambda(s)$ for all $s\in G$. In this case, the map 
$M_\psi$ is necessarily unique and the function $\psi$
is continuous and bounded, with $\norm{\psi}_\infty\leq \norm{M_\psi}$. 
We refer the reader to \cite{DCH}
for these facts and basic results on these Fourier multipliers.

We say that a bounded Fourier multiplier
$\psi$ is completely bounded if the associated  mapping $M_\psi\colon VN(G)\to VN(G)$ 
is completely bounded. We let 
$\M_0(G)$ denote the space of all completely bounded Fourier multipliers
on $VN(G)$ and we equip it with the norm
$$
\norm{\psi}_0 : = \bignorm{M_\psi\colon VN(G)\longrightarrow VN(G)}_{cb},
\qquad \psi\in \M_0(G).
$$
Here $\norm{\,\cdotp}_{cb}$ stands for the completely bounded norm. We refer to
\cite{Pau} for information on completely bounded maps.

Since $\M_0(G)\subset L^\infty(G)$ and $\norm{\psi}_\infty\leq 
\norm{\psi}_0$ for all $\psi\in \M_0(G)$, we may define
$$
\norm{f}_Q : =\sup\Bigl\{
\Bigl\vert\int_G f(t)\psi(t)\, dt\,\Bigr\vert\, :\,
\psi\in \M_0(G),\ \norm{\psi}_0\leq 1\Bigr\},
$$
for all $f\in L^1(G)$. This is a norm on $L^1(G)$. We let 
$Q(G)$ denote the completion of $(L^1(G),\norm{\,\cdotp}_Q)$. 
The following is proved in \cite[Proposition 1.10]{DCH}, see also
\cite[Section 6.2]{Spronk}.

\begin{proposition}\label{3Duality}
The duality pairing $(f,\psi)\mapsto \int f\psi$ on 
$L^1(G)\times \M_0(G)$ extends to a isometric identification
$$
\M_0(G)\simeq Q(G)^*.
$$
\end{proposition}

The following remarkable characterization of completely bounded Fourier multipliers
on $VN(G)$ follows from  \cite[Theorem 3.3]{Spronk} and 
the case $n=1$ of \cite[Theorem 5.3]{Spronk}. It can also be regarded as a combination
of \cite[Lemmas 1 and 2]{Herz}, \cite[Theorem]{BF} and Proposition \ref{2Description} above.

\begin{proposition}\label{3BF}
Let $\psi\colon G\to \Cdb$ be a continuous function and let $C\geq 0$ be a constant. The
following assertions are equivalent.
\begin{itemize}
\item [(i)] $\psi\in \M_0(G)$ and $\norm{\psi}_0\leq C$.
\item [(ii)] For any finite subset $J\subset G$, the family 
$\phi_J :=\{\psi(ts)\,:\, (s,t)\in J^2\}$, regarded as an element of
${\mathcal V}_2(J^2)=\ell^\infty_{J^2}$, satisfies
$$
\nu_2(\phi_J)\leq C.
$$
\end{itemize}
\end{proposition}

Let $\pi\colon G\to B(X)$ be a map. We say that $\pi$ is a representation if $\pi(ts)=\pi(t)\pi(s)$
for all $s,t\in G$ and we say that $\pi$ is bounded if there exists a constant $C\geq 0$
such that $\norm{\pi(t)}\leq C$ for all $t\in G$. We further say that
$\pi$ is strongly continuous if $t\mapsto \pi(t)x$
is continuous (from $G$ into $X$) for all $x\in X$. 

To any bounded strongly continuous representation $\pi\colon G\to B(X)$, one may 
associate a bounded map $\sigma_\pi\colon L^1(G)\to B(X)$ defined by
\begin{equation}\label{3Sigma}
\sigma_\pi(f)=\int_G f(t)\pi(t)\, dt,\qquad f\in L^1(G),
\end{equation}
where the above integral is defined in the strong sense. We note that if $L^1(G)$ is equipped with
convolution, then $\sigma_\pi$ is a homomorphism.

The main result of this subsection is the following.

\begin{theorem}\label{3Gr1} Let $\pi\colon G\to B(X)$ be a $\gamma$-bounded 
strongly continuous representation.
Then the mapping $\sigma_\pi$, defined by (\ref{3Sigma}), uniquely extends to a bounded map
$\sigma_\pi^{Q}\colon Q(G)\to B(X)$, and we have
$$
\norm{\sigma_\pi^{Q}}
\leq\gamma(\pi)^2.
$$
\end{theorem}

\begin{proof}
Let $f\in L^1(G)$. Consider any $x\in X$
and $x^*\in X^*$ and define $\psi\colon G\to \Cdb$ by
$$
\psi(s)=\langle \pi(s)x,x ^*\rangle,\qquad s\in G.
$$
This is a continuous function. 
Fix some finite subset $J\subset G$ and consider $\phi_J$ as in Proposition \ref{3BF}.
Note that we have $\psi(ts)=\langle \pi(t)\pi(s) x,x ^*\rangle$, for all $s,t\in J$.
Hence by Theorem \ref{2Main}, the $\nu_2$-norm of $\phi_J$ is less than or equal to
$\gamma(\pi)^2\norm{x}\norm{x^*}$. 
By Proposition \ref{3BF}, this implies that $\psi\in\M_0(G)$, with
$\norm{\psi}_0\leq\gamma(\pi)^2\norm{x}\norm{x^*}$. 
We may write
$$
\langle \sigma_\pi(f)x,x^*\rangle = \int_G f(s)\psi(s)\, ds.
$$
Appying Proposition \ref{3Duality}, we deduce that 
$$
\bigl\vert\langle \sigma_\pi(f)x,x^*\rangle\bigr\vert
\leq \norm{f}_Q\norm{\psi}_0\leq
\norm{f}_Q\gamma(\pi)^2\norm{x}\norm{x^*}.
$$
Since $x,x^*$ were arbitrary,
this implies that $\norm{\sigma_\pi(f)}\leq \norm{f}_Q\gamma(\pi)^2$,
and the result follows.
\end{proof}

Recall that the reduced $C^*$-algebra of $G$ is defined as
$$
C_\lambda^*(G) :=\overline{\sigma_\lambda(L^1(G))}\subset B(L^2(G)).
$$
It is well-known that if $G$ is amenable, then 
$Q(G)\simeq C_\lambda^*(G)$ isometrically (this follows e.g. from \cite[Corollary 1.8]{DCH}). More precisely, we have
$$
\norm{f}_Q=\norm{\sigma_\lambda(f)},\qquad f\in L^1(G).
$$

Combining this result with Theorem \ref{3Gr1}, we immediately deduce the following result,
which is the main result of \cite{LM1} (see Theorem 4.4 in the latter paper).

\begin{corollary}\label{3Gr2} Let  $G$ be an amenable locally compact group and 
let $\pi\colon G\to B(X)$ be a $\gamma$-bounded strongly continuous representation. Then there
exists a (necessarily unique) bounded operator
$$
w_\pi\colon C^*_\lambda(G)\longrightarrow B(X)
$$
such that $w_\pi\circ \sigma_\lambda = \sigma_\pi$. Moreover 
$w_\pi$ is a homomorphism and $\norm{w_\pi}\leq \gamma(\pi)^2$.
\end{corollary}

\medskip
\subsection{Bounded $C_0$-semigroups}\label{Semigroup1}
Let $\T=(T_t)_{t\geq 0}$ be a bounded $C_0$-semigroup on $X$.
To any $b\in L^1(\Rdb_+)$, one may associate 
an operator $\Gamma(\T,b)\in B(X)$ defined by
$$
\Gamma(\T,b) = \int_{0}^{\infty} b(t)T_t\, dt,
$$
where the above integral is defined in the strong sense. It is plain that
\begin{equation}\label{3Rough}
\norm{\Gamma(\T,b)}\leq\bigl(\sup_{t\geq 0}\norm{T_t}\bigr)\norm{b}_1,
\qquad b\in L^1(\Rdb_+). 
\end{equation}
The mapping
$b\mapsto \Gamma(\T,b)$ is called the Hille-Phillips functional calculus of $\T$. We refer 
the reader to \cite{HP, Paz} and to \cite[Section 3.3]{Haase}
for basics on $C_0$-semigroups and on the Hille-Phillips functional calculus.

In the recent paper \cite{ALZ}, L. Arnold, S. Zadeh and the author showed that
if $X$ is a Hilbert space, then the rough estimate (\ref{3Rough}) can be 
significantly improved, which yields a new functional calculus for generators
of bounded $C_0$-semigroups on Hilbert space (see also \cite{AL}). 
Our aim is to extend the main functional calculus result from \cite{ALZ}
to $\gamma$-bounded $C_0$-semigroups on Banach spaces.

We need definitions from \cite{ALZ}. For consistancy
we will use the notations from this paper, with only one
exception: the above operator $\Gamma(\T,b)$ is called
$\Gamma(A,b)$ in \cite{ALZ}, where $-A$ is the infinitesimal 
generator of $\T$. Indeed, the Hille-Phillips functional calculus of $\T$  can be regarded
as an elementary holomorphic functional calculus associated with $A$,
see \cite[Section 3.3]{Haase} and \cite{AL, ALZ} for more on this.

We let  $S^\infty$ (resp. $S^1$) denote the space of all compact 
operators (resp. all trace class operators) on $\ell^2$. Next we let 
$C_0(\Rdb;S^\infty)$ denote the space of all continuous functions
$f\colon \Rdb\to S^\infty$ which vanish at $\infty$. This is a
Banach space for the norm $\norm{f}_\infty=\sup\{\norm{f(t)}_{S^\infty}\, :\,
t\in\Rdb\}$. Finally we let $H^1(\Rdb;S^1)$ be the Hardy space 
of all $h\in L^1(\Rdb;S^1)$ whose Fourier transform vanishes on $\Rdb_-$. This 
is a closed subspace of $L^1(\Rdb;S^1)$, hence a Banach space for the
induced norm $\norm{\,\cdotp}_1$.

For any $f\in C_0(\Rdb;S^\infty)$ and $h\in H^1(\Rdb;S^1)$, we let $f\ast h\in C_0(\Rdb)$
be  defined by
$$
(f\ast h)(s) =\int_{-\infty}^{\infty} tr\bigl(f(t)h(s-t)\bigr)\, dt,\qquad
s\in\Rdb,
$$
where ${\rm tr}$ denotes the usual trace
on $B(\ell^2)$.
It is plain that $\norm{f\ast h}_\infty\leq\norm{f}_\infty\norm{h}_1$.

We let $\mathcal{A}_{0,S^1}(\Rdb)$ 
be the space of all functions $F
\colon \Rdb \rightarrow \Cdb$ such that there 
exist two sequences $(f_k)_{k\geq 1}$ in 
$C_0(\Rdb;S^\infty)$  
and $(h_k)_{k\geq 1}$  in $H^1(\Rdb;S^1)$ satisfying
\begin{equation}\label{DefA0}
\sum_{k=1}^{\infty} \norm{f_k}_\infty\norm{h_k}_1 < \infty
\qquad\hbox{and}\qquad
F(s) = \sum_{k=1}^{\infty} f_k\ast h_k(s),\quad s\in\Rdb.
\end{equation}
For such a function $F$, we 
set
\begin{equation}\label{DefNormA0}
\norm{F}_{{\mathcal A}_{0,S^1}} = \inf \Bigl\{\sum_{k=1}^{\infty} 
\norm{f_k}_\infty\norm{h_k}_1 \Bigr\},
\end{equation}
where the infimum runs over all sequences  $(f_k)_{k\geq 1}$ in 
$C_0(\Rdb;S^\infty)$  
and $(h_k)_{k\geq 1}$  in $H^1(\Rdb;S^1)$ satisfying (\ref{DefA0}). 
It is clear that $\A_{0,S^1}(\Rdb)\subset C_0(\Rdb)$ and that 
$\norm{F}_\infty\leq \norm{F}_{{\mathcal A}_{0,S^1}}$ for all $F \in \A_{0,S^1}(\Rdb)$.
It is shown in \cite[Proposition 4.1]{ALZ} that 
$\A_{0,S^1}(\Rdb)$ is a Banach algebra for the pointwise multiplication.

For any $b\in L^1(\Rdb)$, let $\widehat{b}\in C_0(\Rdb)$ denote the Fourier
transform of $b$ and let $\widehat{b}(-\,\cdotp)$ denote the function
taking $s$ to  $\widehat{b}(-s)$ for all $s\in\Rdb$. According
to \cite[Lemma 4.4]{ALZ},  the function  $\widehat{b}(-\,\cdotp)$ belongs to 
$\A_{0,S^1}(\Rdb)$ for all $b\in L^1(\Rdb_+)$ and moreover,
$$
\bigl\{\widehat{b}(-\,\cdotp)\, :\, b\in L^1(\Rdb_+)\bigr\}\subset 
\A_{0,S^1}(\Rdb)
$$
is a dense subspace.

We say that  $\T=(T_t)_{t\geq 0}$ is $\gamma$-bounded if the set
$\{T_t\, :\, t\geq 0\}$ is $\gamma$-bounded.  In accordance with Definition \ref{Gamma-app},
we set
$$
\gamma(\T)=\gamma\bigl(\{T_t\, :\, t\geq 0\}\bigr)
$$
in this case. The following result is a generalization of \cite[Theorem 5.2]{ALZ}. We recover the latter by applying
Remark \ref{2Hilb}.

\begin{theorem}\label{3ALZ}
Let $\T=(T_t)_{t\geq 0}$ be a $\gamma$-bounded $C_0$-semigroup on $X$. 
Then there exists a (necessarily unique) bounded homomorphism
$$
\rho_{\mathcal T} \colon \A_{0,S^1}(\Rdb)\longrightarrow B(X)
$$
such that 
$$
\rho_{\mathcal T} \bigl(\widehat{b}(-\,\cdotp)\bigr) =
\int_{0}^\infty b(t)T_t\,dt
$$
for all $b\in L^1(\Rdb_+)$. Moreover,
$\norm{\rho_{\mathcal T}}\leq \gamma(\T)^2$.
\end{theorem}

Before proceeding to the proof, we
need some more background from \cite{ALZ}.
We say that a function $m\in L^\infty(\Rdb_+)$ is a bounded Fourier multiplier on 
$H^1(\Rdb)$ if there exists a (necessarily unique) bounded operator $R_m\colon H^1(\Rdb)\to H^1(\Rdb)$
such that $\widehat{R_m(h)} = m\widehat{h}$ on $\Rdb_+$, for all $h\in H^1(\Rdb)$. We further say
that $m$ is an $S^1$-bounded Fourier multiplier if $R_m\otimes I_{S^1}\colon H^1(\Rdb)\otimes S^1\to H^1(\Rdb)\otimes S^1$
extends to a bounded operator $R_m\overline{\otimes} I_{S^1}\colon H^1(\Rdb;S^1)\to  H^1(\Rdb;S^1)$.
The space $\M_{S^1}(H^1(\Rdb))$ of all  $S^1$-bounded Fourier multipliers $m$ on 
$H^1(\Rdb)$ is a Banach space for the norm $\norm{m}_{{\mathcal M}_{S^1}} :
= \norm{R_m\overline{\otimes} I_{S^1}\colon H^1(\Rdb;S^1)\to  H^1(\Rdb;S^1)}$.
It is proved in \cite[Section 4]{ALZ} that the duality pairing
$$
\bigl\langle \widehat{b}(-\,\cdotp), m\bigr\rangle : = \int_{0}^{\infty} b(t)m(t)\, dt,
\qquad b\in L^1(\Rdb_+),\ m\in \M_{S^1}(H^1(\Rdb)),
$$
extends to an isometric isomorphism
\begin{equation}\label{3Dual}
\M_{S^1}(H^1(\Rdb))\simeq \A_{0,S^1}(\Rdb)^*.
\end{equation}
Indeed, this follows from Proposition 4.2  and Lemma 4.4 in \cite{ALZ}.

Another key result stated as \cite[Theorem 3.1]{ALZ} is that a function $m\in L^\infty(\Rdb_+)$
belongs to $\M_{S^1}(H^1(\Rdb))$ if and only if the function
\begin{equation}\label{3phim}
\phi_m\colon\Rdb_+^2\longrightarrow \Cdb,\qquad \phi_m(s,t)= m(s+t),
\end{equation}
belongs to ${\mathcal V}_2(\Rdb_+^2)$ and that in this case, we have
\begin{equation}\label{3m}
\norm{m}_{{\mathcal M}_{S^1}} = \nu_2(\phi_m).
\end{equation}

\begin{proof}[Proof of Theorem \ref{3ALZ}]
It clearly suffices to show that for any $b\in L^1(\Rdb_+)$, we have
\begin{equation}\label{3Est}
\bignorm{\Gamma(\T,b)}\leq \gamma(\T)^2\bignorm{\widehat{b}(-\,\cdotp)}_{{\mathcal A}_{0,S^1}}.
\end{equation}
Consider $b\in L^1(\Rdb_+)$ and let $x\in X$ and $x^*\in X^*$. Set 
$$
m(t)=\langle T_t(x),x^* \rangle,\qquad t\geq 0,
$$
and define $\phi_m$ as in (\ref{3phim}). Then  we have
$$
\phi_m(s,t) = \langle T_t T_s(x),x^*\rangle,
\qquad s,t\geq 0.
$$
Let us apply 
Theorem \ref{2Main} with $A,B\colon\Rdb_+\to B(X)$ defined by 
$A(t)=B(t)=T_t$. We obtain that $\phi_m$ belongs to ${\mathcal V}_2(\Rdb_+^2)$ and
that
$\nu_2(\phi_m)\leq \gamma(\T)^2\norm{x}\norm{x^*}$. Consequently,
$m$ belongs to $\M_{S^1}(H^1(\Rdb))$ and
$\norm{m}_{{\mathcal M}_{S^1}}\leq \gamma(\T)^2\norm{x}\norm{x^*}$, by (\ref{3m}).

We may write 
$$
\bigl\langle \bigl[\Gamma(\T,b)\bigr](x) , x^*\bigr\rangle = \int_{0}^\infty b(t) m(t)\,dt.
$$
Using (\ref{3Dual}), we obtain
\begin{align*} 
\bigl\vert \bigl\langle \bigl[\Gamma(\T,b)\bigr](x) , x^*\bigr\rangle\bigr\vert & \leq 
\bignorm{\widehat{b}(-\,\cdotp)}_{{\mathcal A}_{0,S^1}}\norm{m}_{{\mathcal M}_{S^1}}\\
& \leq 
\bignorm{\widehat{b}(-\,\cdotp)}_{{\mathcal A}_{0,S^1}}\gamma(\T)^2\norm{x}\norm{x^*}.
\end{align*}
Since $x,x^*$ were arbitrary, this yields (\ref{3Est}).
\end{proof}

\medskip
\subsection{Power bounded operators}\label{PB}
We will finally establish a discrete analogue of Theorem \ref{3ALZ},
which slightly improves \cite[Section 4.5]{A}. The arguments are similar 
to the ones in Subsection \ref{Semigroup1} so we will be brief.

Let $\Tdb\simeq \Rdb/2\pi\Zdb$ denote the unit circle, equipped with its normalized
Haar measure.
For any $n\in\Zdb$, let $e_n\in C(\Tdb)$ be defined by $e_n(t)=e^{int}$ and let 
$$
\P={\rm Span}\{e_n\, :\, n\geq 0\}
$$
be the space of analytic polynomials. 
Let $C(\Tdb;S^\infty)$ denote the space of all continuous functions from $\Tdb$ into $S^\infty$,
equipped with the supremum norm $\norm{\,\cdotp}_\infty$. For any 
$h\in L^1(\Tdb;S^1)$, let $\widehat{h}(n)$, $n\in\Zdb$, denote the Fourier coefficients 
of $h$. The Hardy space $H^1(\Tdb;S^1)$ is the closed subspace 
of all $h\in L^1(\Tdb;S^1)$ such that $\widehat{h}(n)=0$ for all
$n\leq -1$, equipped with the 
$L^1(\Rdb;S^1)$-norm $\norm{\,\cdotp}_1$.  
For any $f\in C(\Tdb;S^\infty)$ and $h\in H^1(\Tdb;S^1)$, we let $f\ast h\in C(\Tdb)$
be  defined by
$$
(f\ast h)(s) =\,\frac{1}{2\pi}\, \int_{-\pi}^{\pi} tr\bigl(f(t)h(s-t)\bigr)\, dt,\qquad
s\in\Rdb.
$$
Then we have $\norm{f\ast h}_\infty\leq\norm{f}_\infty\norm{h}_1$ and as in Subsection \ref{Semigroup1},
we can define $\mathcal{A}_{0,S^1}(\Tdb)$ 
as the space of all functions $F
\colon \Tdb \rightarrow \Cdb$ such that there 
exist two sequences $(f_k)_{k\geq 1}$ in 
$C(\Tdb;S^\infty)$  
and $(h_k)_{k\geq 1}$  in $H^1(\Tdb;S^1)$ satisfying (\ref{DefA0}). This is a
Banach space for the norm given by (\ref{DefNormA0}), where 
the infimum runs over all sequences  $(f_k)_{k\geq 1}$ in 
$C(\Tdb;S^\infty)$  
and $(h_k)_{k\geq 1}$  in $H^1(\Tdb;S^1)$ satisfying (\ref{DefA0}). We clearly have 
$$
\A_{0,S^1}(\Tdb)\subset C(\Tdb)
\qquad\hbox{with}\qquad 
\norm{F}_\infty\leq \norm{F}_{{\mathcal A}_{0,S^1}},\quad F \in \A_{0,S^1}(\Tdb).
$$

The above definition of  $\A_{0,S^1}(\Tdb)$ goes back to the ``Notes and Remarks"  of \cite[Chapter 6]{Pis}, 
where this space is denoted by $\mathcal L$. 
As noticed in this reference, this space originates from  \cite{Peller}, where
it is defined in a different form.
 It turns out that $\A_{0,S^1}(\Tdb)$ is a Banach algebra for the pointwise multiplication,
see the comments after \cite[(6.4)]{Pis}. Moreover it is easy to check that $e_n\in \A_{0,S^1}(\Tdb)$ for all
$n\geq 0$ 
and that 
$$
\P\subset \A_{0,S^1}(\Tdb)
$$
is a dense subspace.

We say that a bounded sequence $m=(m(n))_{n\geq 0}$ of complex numbers
is a bounded Fourier multiplier on 
$H^1(\Tdb)$ if there exists a (necessarily unique) bounded operator $R_m\colon H^1(\Tdb)\to H^1(\Tdb)$
such that $\widehat{R_m(h)}(n) = m(n)\widehat{h}(n)$ for all $h\in H^1(\Tdb)$ and
all $n\geq 0$. We further say
that $m$ is an $S^1$-bounded Fourier multiplier if $R_m\otimes I_{S^1}$
extends to a bounded operator $R_m\overline{\otimes} I_{S^1}\colon H^1(\Tdb;S^1)\to  H^1(\Tdb;S^1)$.
The space $\M_{S^1}(H^1(\Tdb))$ of all  $S^1$-bounded Fourier multipliers $m$ on 
$H^1(\Tdb)$ equipped with the norm $\norm{m}_{{\mathcal M}_{S^1}} :
= \norm{R_m\overline{\otimes} I_{S^1}\colon H^1(\Tdb;S^1)\to  H^1(\Tdb;S^1)}$
is a Banach space. Note that $S^1$-bounded Fourier multipliers on 
$H^1(\Tdb)$ coincide with the completely bounded Fourier multipliers on 
$H^1(\Tdb)$ studied in  \cite[Chapter 6]{Pis}.
It follows from the ``Notes and Remarks"  of \cite[Chapter 6]{Pis} that the 
duality pairing
$$
\langle F, m \rangle : = \sum_{n=0}^\infty \widehat{F}(n)m(n),
\qquad F\in \P,\ m\in \M_{S^1}(H^1(\Tdb)),
$$
extends to an isometric isomorphism
$$
\M_{S^1}(H^1(\Tdb))\simeq \A_{0,S^1}(\Tdb)^*.
$$
Moreover it is proved in \cite[Theorem 6.2]{Pis} that a bounded sequence $m=(m(n))_{n\geq 0}$
belongs to $\M_{S^1}(H^1(\Rdb))$ if and only if the doubly indexed sequence
$$
\bigl((\phi_m(k,l)\bigr)_{k,l\geq 0} \qquad \phi_m(k,l)= m(k+l),
$$
belongs to ${\mathcal V}_2(\Ndb^2)$ and that in this case, we have
$$
\norm{m}_{{\mathcal M}_{S^1}} = \nu_2(\phi_m).
$$

Let $T\in B(X)$. We say that $T$
is power bounded if $(T^n)_{n\geq 0}$ is a bounded sequence and we say that 
$T$ is power $\gamma$-bounded if the set $\{T^n\, :\, n\geq 0\}$ is $\gamma$-bounded. In this case, we set
$$
\gamma(T^\bullet) = \gamma\bigl(\{T^n\, :\, n\geq 0\}\bigr).
$$

The discrete analogue of Theorem \ref{3ALZ} is the following.

\begin{theorem}\label{3Discrete}
Let $T\in B(X)$ be a power $\gamma$-bounded operator. Then there exists a (necessarily unique)
bounded homomorphism 
$$
\rho_{T} \colon \A_{0,S^1}(\Tdb)\longrightarrow B(X)
$$
such that 
$$
\rho_{T} (F) =  \sum_{n=0}^\infty \widehat{F}(n) T^n
$$
for all $F\in \P$. Moreover,
$\norm{\rho_{T}}\leq \gamma(T^\bullet)^2$.
\end{theorem}

\begin{proof}
The proof is entirely similar to the one of  Theorem \ref{3ALZ}. It suffices to show that 
\begin{equation}\label{3Est2}
\Bignorm{\sum_{n=0}^\infty \widehat{F}(n) T^n}\leq  \gamma(T^\bullet)^2 \norm{F}_{{\mathcal A}_{0,S^1}},
\qquad F\in \P.
\end{equation}
Consider $F\in\P$, let $x\in X, x^*\in X^*$ and define  $m=(m(n))_{n\geq 0}$
by $m(n)=\langle T^n(x),x^*\rangle$.
Then $\phi_m(k,l)=\langle T^l T^k(x),x^*\rangle$ for all $k,l\geq 0$. Applying Theorem \ref{2Main}
with $A,B\colon\Ndb\to B(X)$ defined by $A(k)=B(k)=T^k$, we obtain that 
$\phi_m$ belongs to ${\mathcal V}_2(\Ndb^2)$ and hence that $m$ belongs
to $\M_{S^1}(H^1(\Tdb))$, with 
$\norm{m}_{{\mathcal M}_{S^1}}\leq \gamma(T^\bullet)^2\norm{x}\norm{x^*}$.

We may write 
$$
\Bigl\langle \Bigl(\sum_{n=0}^\infty \widehat{F}(n) T^n\Bigr)(x) , x^*\Bigr\rangle = \sum_{n=0}^\infty \widehat{F}(n) m(n),
$$
from which we deduce that
$$
\Bigl\vert \Bigl\langle \Bigl(\sum_{n=0}^\infty \widehat{F}(n) T^n\Bigr)(x) , x^*\Bigr\rangle \Bigr\vert
\leq  \norm{F}_{{\mathcal A}_{0,S^1}}\norm{m}_{{\mathcal M}_{S^1}}\leq
\gamma(T^\bullet)^2 \norm{F}_{{\mathcal A}_{0,S^1}}\norm{x}\norm{x^*}.
$$
Since $x,x^*$ were arbitrary, this yields (\ref{3Est2}).
\end{proof}

\bigskip
\noindent
{\bf Acknowledgements.} 
The author was supported by the ANR project Noncommutative
analysis on groups and quantum groups (No./ANR-19-CE40-0002).

\end{document}